\documentclass[12pt,a4paper]{amsart}
\usepackage{amsmath}
\usepackage{drpack}
\usepackage[english]{babel}
\usepackage{verbatim}
\usepackage[shortlabels]{enumitem}
\usepackage{tikz-cd}
\usepackage{mathtools}

\newtheoremstyle{mio}%
	{}{} 
	{\itshape}{} 
	{\bfseries}{.}{ } 
	{#1 #2\thmnote{~\mdseries(#3)}} 
\theoremstyle{mio}
\newtheorem{teor}{Theorem}[section]
\newtheorem{cor}[teor]{Corollary}
\newtheorem{prop}[teor]{Proposition}
\newtheorem{lemma}[teor]{Lemma}
\newtheorem{defin}[teor]{Definition}

\newtheoremstyle{definition2}%
	{}{} 
	{}{} 
	{\bfseries}{.}{ } 
	{#1 #2\thmnote{\mdseries~ #3}} 
\theoremstyle{definition2}
\newtheorem{ex}[teor]{Example}
\newtheorem{oss}[teor]{Remark}

\newtheorem{congett}{Conjecture}

\newcommand{\insid}{\mathcal{I}}
\newcommand{\insstar}{\mathrm{Star}}

\newcommand{\inssubmod}{\mathbf{F}}
\newcommand{\insmult}{\mathrm{Mult}}

\newcommand{\reg}{\mathrm{reg}}

\newcommand{\princ}{\ast}
\newcommand{\ordine}{\preceq}
\newcommand{\ordinestretto}{\prec}

\newcommand{\rc}{\color{red}}
\newcommand{\nz}{\bullet}
\newcommand{\qbin}[2]{\binom{#1}{#2}_q}

\title[Asymptotic for the number of star operations]{Asymptotic for the number of star operations on one-dimensional Noetherian domains}
\author{Dario Spirito}
\date{\today}
\address{Dipartimento di Matematica e Fisica, Universit\`a degli Studi ``Roma Tre'', Roma, Italy}
\curraddr{Dipartimento di Matematica, Universit\`a di Padova, Padova, Italy}
\email{spirito@mat.uniroma3.it\\spirito@math.unipd.it}

\begin{document}
\begin{abstract}
We study the set of star operations on local Noetherian domains $D$ of dimension $1$ such that the conductor $(D:T)$ (where $T$ is the integral closure of $D$) is equal to the maximal ideal of $D$. We reduce this problem to the study of a class of closure operations (more precisely, multiplicative operations) in a finite extension $k\subseteq B$, where $k$ is a field, and then we study how the cardinality of this set of closures vary as the size of $k$ varies while the structure of $B$ remains fixed.
\end{abstract}

\maketitle

\section{Introduction}
Star operations are a class of closure operations on the set of (fractional) ideals of an integral domain that has been first studied by Krull \cite{krull_breitage_I-II} and Gilmer \cite[Chapter 5]{gilmer}. Along the years, they have been used, for example, to study and generalize factorization properties (for example, with the characterization of unique factorization domains among Krull domain through the $t$-class group by Samuel \cite{samuel_factoriel}) and to find Pr\"ufer-like classes of integral domains (for example P$v$MDs; see for example \cite[Section 3]{anderson_GCDgauss_2000} for an overview).

More recently, Houston, Mimouni and Park have been studying the set $\insstar(D)$ of star operations on an integral domain $D$, attempting to characterize when this set is finite and, in this case, to calculate its cardinality \cite{twostar,houston_noeth-starfinite,hmp_finite,starnoeth_resinfinito}. They analyzed in particular the Noetherian case, showing that $|\insstar(D)|=\infty$ when $D$ has dimension greater than $1$ \cite[Theorem 2.1]{houston_noeth-starfinite}, showing how to reduce to the local case \cite[Theorem 2.3]{houston_noeth-starfinite} and calculating the cardinality of $\insstar(D)$ when $\ell_D(T/\mathfrak{m}_D)\leq 3$, where $T$ is the integral closure of $D$ and $\mathfrak{m}_D$ the maximal ideal of $D$ \cite[Theorem 3.1]{houston_noeth-starfinite}. Further cases have been considered, for example, in \cite{houston_noeth-starfinite} (for infinite residue field), in \cite{cardinality_pvd,pvd-star} (for pseudo-valuation domains), in \cite{kunz-star} (for Kunz domains) and \cite{white-tesi-sgr} (for some numerical semigroup rings).

Beyond star operations, that are two other classes of closure operations that are in use in commutative algebra: semiprime operations, defined on the set of integral ideals of an arbitrary ring, and semistar operations, defined on the set of all submodules of the quotient field of an integral domain. There are several connections between the classes of semiprime, star and semistar operations; indeed, most definitions and many properties can be stated in essentially the same way, with the main difference being often in how they must be phrased to account for the different partially ordered set on which the closures are defined.

To unify the treatment of these classes of closure operations, the paper \cite{multiplicative} introduced the concept of \emph{multiplicative operations}: these are a class of closure operations that can be defined in any ring extension $A\subseteq B$ over any set $\mathcal{G}$ of $A$-submodules of $B$, and their definition is flexible enough to cover (for suitable choices of $A$, $B$ and $\mathcal{G}$) all three classical cases. Furthermore, multiplicative operations enjoy some functorial properties: while these are usually a staple of many semiprime operations, for star and semistar operations they are very rare, especially due to the fact that quotienting an integral domain disrupts its quotient field. (There are some limited exceptions: see \cite{fontana-park} for properties of star operations along pullbacks and \cite{starloc2} for an application to Pr\"ufer domains.) In the special case of local Noetherian domains of dimension $1$, multiplicative operations allow to bypass this problem by considering only the submodules contained in the integral closure $T$ of the starting domain $D$ (which contain all the needed data), and then by quotienting $T$ over the conductor $(D:T)$ (see \cite[Section 6]{multiplicative} and the discussion after Definition \ref{defin:star} below). In particular, this reduces the problem of finding all star operations on $D$ to the study of the Artinian ring $T/(D:T)$.

In this paper, we continue the study initiated in Sections 6 and 7 of \cite{multiplicative} on this case, in particular concentrating on the case where the conductor $(D:T)$ is equal to the maximal ideal $\mathfrak{m}_D$ of $D$; in the Artinian setting, this case correspond to the study of a particular set of multiplicative operations defined on a finite extension $k\subseteq B$, where $k$ is a field. As we are interested in cardinality problems, we assume throughout the paper that $k$ is a finite field of cardinality $q$: in particular, we are interested in what happens when the structure of $B$ is ``fixed'' (see Section \ref{sect:Artin} for a more precise definition) while $q$ changes, that is, we are interested in the cardinality of the set $\insstar(D)$ of star operations on $D$ as a function of $q$, and especially in understanding how fast the growth of $|\insstar(D)|$ is.

The aim of this paper is to introduce and give some evidence to the following conjecture: if the structure of $B$ is fixed, then the function $q\mapsto\log_q\log_2|\insstar(D)|$ has always a limit as $q\to\infty$, and this limit only depend on the length of $B=T/\mathfrak{m}_D$ as a $D$-module. While we are not able to prove this conjecture in full generality, we show polynomial bounds on the upper and lower limits (respectively, Theorem \ref{teor:up-numstar} and Proposition \ref{prop:liminf}, summarized in Theorem \ref{teor:liminfsup}) and we prove the conjecture in full for $\ell_D(T/\mathfrak{m}_D)\leq 4$ (Proposition \ref{prop:ns3} and Theorem \ref{teor:n4}).

\section{Closure operations}
Let $(\mathcal{P},\leq)$ be a partially ordered set. A \emph{closure operation} on $\mathcal{P}$ is a map $c:\mathcal{P}\longrightarrow\mathcal{P}$ such that, for every $x,y\in\mathcal{P}$:
\begin{itemize}
\item $x\leq c(x)$;
\item if $x\leq y$, then $c(x)\leq c(y)$;
\item $c(c(x))=c(x)$.
\end{itemize}
If $c,d$ are closure operations on $\mathcal{P}$, we write $c\leq d$ if $c(x)\leq d(x)$ for every $x\in\mathcal{P}$.

Any closure operation $c$ is uniquely determined by the set $\mathcal{P}^c:=\{x\in\mathcal{P}\mid c(x)=c\}$ of its fixed points, i.e., two closure operations $c$ and $d$ are equal if and only if $\mathcal{P}^c=\mathcal{P}^d$. Furthermore, $c\leq d$ if and only if $\mathcal{P}^c\supseteq\mathcal{P}^d$.

Let $A\subseteq B$ be a ring extension, and let $\mathcal{G}$ be a set of $A$-submodules of $B$. A \emph{multiplicative operation} \cite{multiplicative} on $(A,B,\mathcal{G})$ is a closure operation $\star$ on $\mathcal{G}$, $I\mapsto I^\star$, such that
\begin{equation*}
(I:b)^\star\subseteq(I^\star:b)
\end{equation*}
for all $I\in\mathcal{G}$ and all $b\in B$ such that $(I:b)\in\mathcal{G}$. We denote the set of multiplicative operations on $(A,B,\mathcal{G})$ by $\insmult(A,B,\mathcal{G})$.

If $D$ is an integral domain with quotient field $K$, a \emph{star operation} on $D$ is a closure operation $\star$ on the set $\insfracid(D)$ of fractional ideals of $D$ such that, for all $I\in\insfracid(D)$ and all $x\in K$, we have:
\begin{itemize}
\item $D=D^\star$;
\item $(xI)^\star=x\cdot I^\star$.
\end{itemize}
We denote by $\insstar(D)$ the set of star operations on $D$. The restriction from $\insfracid(D)$ to the set $\insid(D)^\nz$ of all nonzero ideals of $D$ gives an isomorphism of ordered sets between $\insstar(D)$ and $\insmult(D,K,\insid(D)^\nz)$ \cite[Proposition 3.4]{multiplicative}.

\section{Artinian rings}\label{sect:Artin}
Throughout the paper, we shall denote by $k$ the finite field of cardinality $q$ and by $B$ a finite $k$-algebra that is a principal ideal ring; we also write $\ins{F}_{q^e}$ for the field of cardinality $q^e$, i.e., for the extension of $k$ of degree $e$. In particular, $\ins{F}_q=k$.

We can write $B$ as a direct product $B_1\times\cdots\times B_t$, where each $B_i$ is a local $k$-algebra; by \cite[Theorem 8]{struct-PIR} $B_i$ is isomorphic to $\ins{F}_{q^{e_i}}[X]/(X^{f_i})\simeq \ins{F}_{q^{e_i}}[[X]]/(X^{f_i})$ for some positive integers $e_i,f_i$. We will always consider $k$ as a subring of $B$ through the diagonal embedding; we shall also sometimes consider $k$ as a subring of $B_i$ by the obvious embedding of $k$ into $\ins{F}_{q^{e_i}}[X]/(X^{f_i})$.

It is easy to see that the quantities $e_i,f_i$ are linked to the length of $B$.
\begin{prop}\label{prop:eifi}
Preserve the notation above. Then:
\begin{enumerate}[(a)]
\item $f_i=\ell_{B_i}(B_i)=\ell_B(B_i)$;
\item $L_i$ is the residue field of $B_i$;
\item $\ell_A(B)=\sum_ie_if_i=\ell_D(T/(D:T))$.
\end{enumerate}
\end{prop}
\begin{proof}
Straightforward.
\end{proof}
Let $\mathbf{s}:=\{(e_1,f_1),\ldots,(e_t,f_t)\}$ be a sequence of pairs of positive integers. We call the ring
\begin{equation*}
B_\mathbf{s}:=\frac{\ins{F}_{q^{e_1}}[X]}{(X^{f_1})}\times\cdots\times\frac{\ins{F}_{q^{e_k}}[X]}{(X^{f_k})}
\end{equation*}
the \emph{Artinian ring associated to $\mathbf{s}$}. We denote by $n(\mathbf{s})$ the sum $\sum_ie_if_i=\ell_k(B_\mathbf{s})$.

\begin{oss}
The methods we will use actually hold in the more general setting of a ring extension $A\subseteq B$, where $A$ and $B$ are Artinian rings and $B$ is a finite $A$-module; the restriction to $A=k$, however, allows to simplify the notation. Furthermore, many proofs do not really use the hypothesis that $k$ is finite if not to guarantee that the quantities involved are finite.
\end{oss}

Let now $k,B$ as above. We define:
\begin{itemize}
\item $\mathcal{F}_0(k,B):=\{I\in\inssubmod_k(B)\mid IB=B\}$;
\item $\mathcal{F}_\reg(k,B):=\{I\in\inssubmod_k(B)\mid I\text{~contains a regular element of~}B\}$;
\item $\mathcal{F}_1(k,B):=\{I\in\inssubmod_k(B)\mid 1\in I\}$.
\end{itemize}

These sets are related in the following way.
\begin{prop}\label{prop:insiemiF}
Preserve the notation above.
\begin{enumerate}[(a)]
\item\label{prop:insiemiF:contenimenti} $\insfracid_1(k,B)\subseteq\insfracid_\reg(k,B)\subseteq\insfracid_0(k,B)$.
\item\label{prop:insiemiF:0reg} If $|\Max(B)|\leq q$, then $\insfracid_0(A,B)=\insfracid_\reg(A,B)$.
\item\label{prop:insiemiF:01} If $\insfracid_0(A,B)=\insfracid_\reg(A,B)$, then the restriction map from $\insmult(k,B,\insfracid_0(k,B))$ to $\insmult(k,B,\insfracid_1(k,B))$ is an order isomorphism.
\end{enumerate}
\end{prop}
\begin{proof}
\ref{prop:insiemiF:contenimenti} is obvious. 

\ref{prop:insiemiF:0reg} Let $I\in\insfracid_0(k,B)$. If $I$ is not regular, then $I$ is contained in the union of the maximal ideals of $B$; by \cite[Lemma 3.6]{houston_noeth-starfinite}, it follows that $I$ is contained in some maximal ideal of $B$. But this would imply $1\notin IB$, a contradiction.

\ref{prop:insiemiF:01} Let $I\in\insfracid_0(k,B)$. By hypothesis, $I$ contains a regular element $u$ of $B$, which is a unit since $B$ is Artinian; in particular, $u^{-1}I=(I:u)\in\insfracid_1(k,B)\subseteq\insfracid_0(k,B)$. Hence, $(u^{-1}I)^\star\subseteq u^{-1}I^\star$ for every multiplicative operation $\star$ on $(k,B,\insfracid_0(k,B))$; since we can do the same with the unit $v:=u^{-1}$, it follows that $(u^{-1}I)^\star=u^{-1}I^\star$. Hence, every $\star\in\insmult(k,B,\insfracid_0(k,B))$ is uniquely determined by its action on $\insfracid_1(A,B)$, and so the restriction map is injective; furthermore, every $\sharp\in\insmult(A,B,\insfracid_1(A,B))$ can be extended to $\insfracid_0(A,B)$ by setting $I^\star:=u(u^{-1}I)^\star$ when $u\in I$, and thus the map is also surjective.
\end{proof}

Note that the equality $\insfracid_0(k,B)=\insfracid_\reg(k,B)$ may indeed fail: for example, if $k=\ins{F}_2$ and $B=k^3$ the set $V:=\{(0,0,0),(1,1,0),(1,0,1),(0,1,1)\}$ is a $k$-module in $\insfracid_0(A,B)$, but it is not regular.

\begin{defin}\label{defin:star}
A \emph{star operation} on $(k,B)$ is a multiplicative operation $\star$ on $(k,B,\mathcal{F}_0(A,B))$ such that $k^\star=k$. We denote their set by $\insstar(k,B)$.
\end{defin}

The terminology is justified by the following situation: let $(D,\mathfrak{m})$ be a Noetherian local integral domain of dimension $1$, and let $T$ be its integral closure. Suppose that $T$ is finite over $D$, and suppose that $(D:T)=\mathfrak{m}$. By \cite[Theorem 6.4 and Corollary 6.5]{multiplicative}, there is an order-preserving bijection between the set of star operations on $D$ and the set of star operations on $(D/\mathfrak{m},T/\mathfrak{m})$: in particular, $\insstar(D)$ depends only on the extension $k\subseteq T/\mathfrak{m}$ of Artinian rings, i.e., from the structure of the Artinian ring $T/\mathfrak{m}$ as a $k$-algebra. Such a ring exists for every choice of $\mathbf{s}$, as the next proposition shows.
\begin{prop}\label{prop:portasu}
Let $\mathbf{s}$ be a sequence of pairs of positive integers. Then, there is a local Noetherian domain $D_\mathbf{s}$ of dimension $1$ with integral closure $T$ such that $(D_\mathbf{s}:T)=\mathfrak{m}$ and $T/\mathfrak{m}\simeq B_\mathbf{s}$.
\end{prop}
\begin{proof}
Let $L_1,\ldots,L_t$ be the residue fields of $B_\mathbf{s}$ and, for every $i$, let $f_i(X)\in k[X]$ be an irreducible polynomial with splitting field $L_i$. In the polynomial ring $k[X_1,\ldots,X_t]$, let $P_i$ be the prime ideal generated by $f_i(X_i)$, and let $S:=k[X_1,\ldots,X_t]\setminus\bigcup_iP_i$. Then, $T:=S^{-1}[X_1,\ldots,X_t]$ is a principal ideal domain with $t$ maximal ideals, $P_1T,\ldots,P_tT$, such that the residue field of $P_iT$ is $L_i$.

Let $I:=P_1^{e_1}\cdots P_t^{e_t}T$. Then, $T/I$ is isomorphic to $B_\mathbf{s}$; if $\pi$ is the quotient $T\longrightarrow B_\mathbf{s}$, defining $D_\mathbf{s}:=\pi^{-1}(k)$ we have the ring we were looking for.
\end{proof}

Note that, while $B_\mathbf{s}$ can be defined in a ``canonical'' way, there is a lot more freedom in defining $D_\mathbf{s}$, since it may be possible to use a polynomial ring with fewer indeterminates (for example, if each $L_i$ has a different degree, then it is sufficient to consider $T$ as a localization of $k[X]$).

\begin{defin}
Let $\mathbf{s}$ be a sequence of pairs of positive integers. We denote by $\Lambda(\mathbf{s},q)$ the cardinality of $\insstar(k,B_\mathbf{s})$, where $q:=|k|$.
\end{defin}

If $\mathbf{s}$ is fixed, we can expect the function $q\mapsto\Lambda(\mathbf{s},q)$ to be increasing, since as $q$ grows the ring $B_\mathbf{s}$ contains more and more subspaces and thus more and more closure operations. Indeed, this is exactly what happens for $n(\mathbf{s})>3$, while for $n(\mathbf{s})\leq 3$ the map $\Lambda(\mathbf{s},q)$ is constant (see \cite{houston_noeth-starfinite} and the proof of Proposition \ref{prop:ns3} below). To study how fast it grows, we introduce the following function.
\begin{defin}
Let $\mathbf{s}$ be a sequence of pairs of positive integers. We set
\begin{equation*}
\theta(\mathbf{s},q):=\log_q\log_2\Lambda(\mathbf{s},q).
\end{equation*}
\end{defin}

It is not at all obvious that this is a good choice. However, we have the following.
\begin{prop}\label{prop:gamman-PVD}
For every $n\inN$, we have 
\begin{equation*}
\lim_{q\to\infty}\theta((n,1),q)=\begin{cases}
\frac{(n-2)^2}{4} & \text{if~}$n$\text{~is even},\\
\frac{(n-1)(n-3)}{4}  & \text{if~}$n$\text{~is odd.}
\end{cases}
\end{equation*}
\end{prop}
\begin{proof}
If $\mathbf{s}=\{(n,1)\}$, the ring $B:=B_{\mathbf{s}}$ is just the degree $n$ extension of $k$. Let $\gamma_n$ be the number on the right hand side in the formula of the statement.

Suppose $n$ is even. By \cite[Theorem 3.7]{pvd-star}, for every $\epsilon>0$ we have, for large $q$,
\begin{equation*}
q^{\gamma_n}\leq\log_2|\insstar(k,B)|\leq(1+\epsilon)q^{\gamma_n};
\end{equation*}
taking the logarithm in base $q$ we have
\begin{equation*}
\gamma_n\leq\theta((n,1),q)\leq\gamma_n+\log_q(1+\epsilon)
\end{equation*}
and the claim follows by taking the limit. If $n$ is odd, instead of $1+\epsilon$ we have $2+\epsilon$, but the same reasoning applies.
\end{proof}

Another case when we can calculate the limit of $\theta(\mathbf{s},q)$ is for low $n(\mathbf{s})$.
\begin{prop}\label{prop:ns3}
Let $\mathbf{s}$ be a sequence of pairs of positive integers. Then:
\begin{enumerate}
\item if $n(\mathbf{s})\leq 2$, then $\theta(\mathbf{s},q)=-\infty$ for every $q$;
\item if $n(\mathbf{s})=3$, then
\begin{equation*}
\lim_{q\to\infty}\theta(\mathbf{s},q)=0
\end{equation*}
\end{enumerate}
\end{prop}
\begin{proof}
By Proposition \ref{prop:portasu}, we can find a local Noetherian domain $(D,\mathfrak{m})$ of dimension $1$ such that $\insstar(D)\simeq\insstar(k,B)$; let $T$ be the integral closure of $D$. Then, $\ell(T/D)=n(\mathbf{s})$.

If $\ell(T/D)\leq 2$, then by \cite[Theorem 6.3]{bass_ubiqGor} or \cite[Theorem 3.8]{matlis-reflexive}  $\insstar(D)$ is a singleton, and thus
\begin{equation*}
\theta(\mathbf{s},q)=\log_q\log_21=\log_10=-\infty.
\end{equation*}

If $n(\mathbf{s})=3$, then by \cite[Theorem 3.1]{houston_noeth-starfinite} the cardinality of $\insstar(D)$ does not depend on $k=D/\mathfrak{m}$. Hence, $\log_2(\Lambda(\mathbf{s},q))$ is constant in $q$, and so $\theta(\mathbf{s},q)=\log_q\log_2(\Lambda(\mathbf{s},q))$ goes to $0$.
\end{proof}

In view of these two cases, we advance the following conjectures.
\begin{congett}\label{cong:limite}
For every sequence of pairs of positive integers $\mathbf{s}$, the function $q\mapsto\theta(\mathbf{s},q)$ has a limit when $q\to\infty$. Furthermore, 
\end{congett}

\begin{congett}\label{cong:ns}
The limit of $\theta(\mathbf{s},q)$ as $q\to\infty$ depends only on $n(\mathbf{s})$, that is, if $n(\mathbf{s})=n(\mathbf{t})$ then $\theta(\mathbf{s},q)$ and $\theta(\mathbf{t},q)$ have the same limit as $q\to\infty$.
\end{congett}

It is to be noted that Propositions \ref{prop:gamman-PVD} and \ref{prop:ns3} are quite flimsy evidence for these conjectures; in particular, Proposition \ref{prop:ns3} would also be true if, instead of the very peculiar function $q\mapsto\log_q\log_2\Lambda(\mathbf{s},q)$, we would have taken any function $q\mapsto\zeta(q)$ with limit $0$ as $q\to\infty$. In the rest of the paper, we will give some justification for this choice by showing that, as $q\to\infty$, the limit inferior and superior of $\theta(\mathbf{s},q)$ can be bounded by functions in $n(\mathbf{s})$ with a polynomial growth (Theorem \ref{teor:liminfsup}) and that Conjecture \ref{cong:ns} is true also for $n=4$.

\section{The number of subspaces}
Let $A\subseteq B$ be a ring extension and let $\mathcal{G}\subseteq\inssubmod_B(A)$. Then, a multiplicative operation $\star$ on $(A,B,\mathcal{G})$ is uniquely determined by the set $\mathcal{G}^\star:=\{I\in\mathcal{G}\mid I=I^\star\}$ of its closed ideals. In particular, there is an obvious bound $|\insmult(A,B,\mathcal{G})|\leq 2^{|\mathcal{G}|}$. In this section, we use this fact to bound the number of star operations on $(k,B)$.

Given a finite field $k$ of cardinality $q$, and a $k$-vector space $V$ of dimension $n$, we denote by $Z(q,n)$ the number of vector subspaces of $V$.

\begin{lemma}\label{lemma:numvsp}
For every $n$ and every $\epsilon>0$, there is a $q(n,\epsilon)$ such that
\begin{equation*}
Z(q,n)\leq \begin{cases}
n\cdot q^{n^2/4+\epsilon(n/2)} & \text{if~}n\text{~is even},\\
n\cdot q^{(n^2-1)/4+\epsilon((n+1)/2)} & \text{if~}n\text{~is odd}\\
\end{cases}
\end{equation*}
for all $q\geq q(n,\epsilon)$.
\end{lemma}
\begin{proof}
Fix $\epsilon>0$.

Let $Z_t(q,n)$ be the number of $t$-dimensional subspaces of $V$. Then (see e.g. \cite[Proposition 1.3.18]{stanley-EC1} or \cite[Chapter 13, Proposition 2.1]{handbook-combinatorics}), $Z_t(q,n)$ is given by the \emph{$q$-binomial coefficient}
\begin{equation*}
\qbin{n}{t}:=\frac{(q^n-1)\cdots(q^{n-t+1}-1)}{(q^t-1)\cdots(q-1)}.
\end{equation*}
For large $q$, we have 
\begin{equation*}
q^{n-t+j}-1\leq q^{n-t+\epsilon}(q^j-1)
\end{equation*}
since the dominant term on the right hand side is $q^{n-t+j+\epsilon}$. Hence,
\begin{equation*}
\qbin{n}{t}=\prod_{j=1}^{n-t}\frac{q^{n-t+j}-1}{q^j-1}\leq\prod_{j=1}^{t}q^{n-t+\epsilon}=q^{t(n-t+\epsilon)}
\end{equation*}
The maximum of the function $t\mapsto t(n-t+\epsilon)$ is in $t_0:=\frac{n+\epsilon}{2}$; checking the integers closer to $t_0$, we obtain that the maximum is in $\frac{n}{2}$ when $n$ is even and in $\frac{n+1}{2}$ when $n$ is odd.

Therefore, if $n$ is even we have
\begin{equation*}
Z(q,n)=\sum_{t=0}^nZ_t(q,n)\leq\sum_{t=1}^nq^{(n/2)(n-(n/2)+\epsilon)}\leq nq^{n^2/4+\epsilon n/2}
\end{equation*}
(where we can forget the zero subspace since the inequalities are not tight). Analogously, if $n$ is odd,
\begin{equation*}
Z(q,n)=\sum_{t=0}^nZ_t(q,n)\leq\sum_{t=1}^nq^{((n+1)/2)(n-((n+1)/2)+\epsilon)}\leq nq^{(n^2-1)/4+\epsilon (n+1)/2}.
\end{equation*}
The claim is proved.
\end{proof}

\begin{teor}\label{teor:up-numstar}
Let $\mathbf{s}$ be a sequence of pairs of positive integers and let $n:=n(\mathbf{s})$. Then,
\begin{equation*}
\limsup_{q\to\infty}\theta(\mathbf{s},q)\leq\begin{cases}
\frac{n(n-2)}{4} & \text{if~}n\text{~is even},\\
\frac{(n-1)^2}{4} & \text{if~}n\text{~is odd}.
\end{cases}
\end{equation*}
\end{teor}
\begin{proof}
By parts \ref{prop:insiemiF:0reg} and \ref{prop:insiemiF:01} of Proposition \ref{prop:insiemiF}, for large $q$ any star operation on $(k,B_\mathbf{s})$ is a multiplicative operation on $(k,B,\mathcal{F}_1(k,B))$, and thus $|\insstar(k,B_\mathbf{s})|\leq 2^{|\mathcal{F}_1(k,B_\mathbf{s})|}$, i.e.,
\begin{equation*}
\theta(\mathbf{s},q)\leq\log_q|\mathcal{F}_1(k,B_\mathbf{s})|.
\end{equation*}
The number of $k$-subspaces of $B_\mathbf{s}$ containing $1$ is the number of $k$-subspaces of a vector space of dimension $n-1$, i.e., $|\mathcal{F}_1(k,B_\mathbf{s})|\leq Z(q,n-1)$. Hence, for $n$ even,
\begin{align*}
\limsup_{q\to\infty}\theta(\mathbf{s},q)\leq& \limsup_{q\to\infty}\log_q((n-1)\cdot q^{((n-1)^2-1)/4+\epsilon((n-1+1)/2)})\leq\\
\leq &  \limsup_{q\to\infty}\left(\frac{(n-1)^2-1}{4}+\epsilon\frac{n}{2}\right)=\frac{n^2-2n}{4}=\frac{n(n-2)}{4}
\end{align*}
using Lemma \ref{lemma:numvsp} on the odd number $n-1$ and since $\epsilon>0$ is arbitrary. The same calculation for $n$ odd gives $\limsup_{q\to\infty}\theta(\mathbf{s},q)\leq\frac{(n-1)^2}{4}$.
\end{proof}

The two quadratic functions that bound $\limsup_{q\to\infty}\theta(\mathbf{s},q)$ are rather close to the quadratic functions that give the limit of $\theta((n,1),q)$ in Proposition \ref{prop:gamman-PVD}: for example, if $n$ is even, then the difference between the upper bound and the limit of the case $(n,1)$ is just $\frac{n(n-2)}{4}-\frac{(n-2)^2}{4}=\frac{n-2}{4}$, which is linear instead of quadratic. 

\section{Lower bounds}
\begin{defin}
Let $I\in\insfracid_0(k,B)$. The \emph{principal star operation generated by $I$} is the largest star operation that closes $I$; we denote it by $\princ_I$. If $\mathcal{A}\subseteq\insfracid_0(k,B)$, the star operation \emph{generated by $\mathcal{A}$}, denoted by $\princ_\mathcal{A}$, is the largest star operation that closes every element of $\mathcal{G}$; equivalently, $\princ_\mathcal{A}=\inf\{\princ_I\mid I\in\mathcal{A}\}$.
\end{defin}

The definition given above is slightly different from the definition given in \cite[Section 4]{multiplicative} since we are imposing that $\princ_I$ also closes $k$ (as we want $\princ_I$ to be a star operation and not only a multiplicative operation). However, the two definitions are actually very close.
\begin{prop}\label{prop:JstarI}
Let $I,J\in\insfracid_1(k,B)$. Then,
\begin{equation*}
J^{\princ_I}=\begin{cases}
k & \text{if~}J=k\\
\bigcup\{(I:b)\mid bJ\subseteq I\} & \text{otherwise}.
\end{cases}
\end{equation*}
\end{prop}
\begin{proof}
Let $v(I)$ be the multiplicative operation generated by $I$ on $(k,B,\insfracid_1(k,B))$, as in \cite[Definition 4.2]{multiplicative}. Then, $\princ_I$ is just the infimum of $v(k)$ and $v(I)$, and thus
\begin{equation*}
J^{\princ_I}=J^{v(k)}\cap J^{v(I)}.
\end{equation*}
By the proof of \cite[Lemma 4.4]{multiplicative}, we have $J^{v(I)}=\bigcup\{(I:b)\mid bJ\subseteq I\}$. On the other hand, if $J=k$ then clearly $J^{v(k)}=k$; if $J\neq k$, then $(k:J)=0$, since if $tJ\subseteq k$ then $t\in k$ (since $1\in J$) but since $\dim_kJ>1$ the only possibility is $t=0$. Hence, again by \cite[Lemma 4.4]{multiplicative},
\begin{equation*}
J^{v(k)}=(k:(k:J))=(k:0)=B.
\end{equation*}
The claim is proved.
\end{proof}

Let $\sim$ be the equivalence relation defined by $I\sim J$ if and only if $\princ_I=\princ_J$. If $\mathcal{G}\subseteq\insfracid_0(k,B)$, we denote by $[\mathcal{G}]$ the set of equivalence classes containing at least one element of $\mathcal{G}$. By definition, $|\insstar(k,B)|\geq|[\mathcal{F}_0(k,B)]|$; however, this bound is too weak, since it is not exponential in the cardinality of $[\mathcal{F}_0(k,B)]$, which is typically polynomial in $q$. Hence, we need a stronger situation.
\begin{lemma}\label{lemma:exponential}
Let $\mathcal{G}\subseteq\mathcal{F}_0(k,B)$ be a subset such that, for every $I\in\mathcal{G}$, we have
\begin{equation*}
I^{\princ_{\mathcal{G}\setminus\{I\}}}\neq I.
\end{equation*}
Then, $|\insstar(k,B)|\geq 2^{|\mathcal{G}|}$.
\end{lemma}
Note that the hypothesis implies that no two distinct elements of $\mathcal{G}$ are equivalent under $\sim$.
\begin{proof}
Let $\mathcal{A}\neq\mathcal{B}$ two subsets of $\mathcal{G}$. Then, without loss of generality there is an $I\in\mathcal{A}\setminus\mathcal{B}$: we have $I^{\princ_\mathcal{A}}\subseteq I^{\princ_I}= I$, so that $I^{\princ_\mathcal{A}}=I$, while
\begin{equation*}
I^{\princ_{\mathcal{B}}}\supseteq I^{\princ_{\mathcal{G}\setminus\{I\}}}\supsetneq I
\end{equation*}
by hypothesis and since $\mathcal{B}\subseteq\mathcal{G}\setminus\{I\}$. Therefore, $\princ_\mathcal{A}\neq\princ_\mathcal{B}$. It follows that every subset of $\mathcal{G}$ generates a different star operation, and so $|\insstar(k,B)|\geq 2^{|\mathcal{G}|}$.
\end{proof}

We now want to apply this criterion. We start from the local case.
\begin{prop}\label{prop:down-numstar-loc}
Let $e,f$ be positive integers and let $n:=ef$. Then,
\begin{equation*}
\theta((e,f),q)\geq n-3.
\end{equation*}
for every $q$.
\end{prop}
\begin{proof}
If $n\leq 3$, then $\theta((e,f),q)\geq 0\geq n-3$ and we are done. Suppose $n\geq 4$. By definition, $\theta((e,f),q)$ is equal to the number of star operations on $(k,B)$, where $B:=L[X]/(X^f)$ for some extension $L$ of $k$ and some $f\geq 1$. 

For every $\alpha\in B\setminus k$, consider the submodule $I(\alpha):=\langle 1,\alpha\rangle$. Every such submodule can be obtained in $q^2-q$ ways (through all the $\beta\in I(\alpha)\setminus k$) and thus there are exactly $(q^n-q)/(q^2-q)=(q^{n-1}-1)/(q-1)$ different submodules. Let $\mathcal{I}$ be the set of the ideals $I(\alpha)$ not containing $X^{f-1}$.

We want to consider $I(\alpha)^{\princ_{I(\beta)}}$. By Proposition \ref{prop:JstarI}, we have
\begin{equation*}
I(\alpha)^{\princ_{I(\beta)}}=\bigcap\{(I(\beta):\gamma)\mid I(\alpha)\subseteq(I(\beta):\gamma)\}.
\end{equation*}
Suppose $I(\alpha)\subseteq(I(\beta):\gamma)$: then, if $\gamma$ is a unit of $B$ we have $I(\alpha)\subseteq\gamma^{-1}I(\beta)$, and the two ideals must be equal since they are both $2$-dimensional $k$-subspaces of $B$. On the other hand, if $\gamma$ is not a unit then $\gamma X^{f-1}=0$ and thus $X^{f-1}\in(I(\beta):\gamma)$. Therefore, we have two cases:
\begin{itemize}
\item if $I(\alpha)\subseteq(I(\beta):\gamma)$ for some unit $\gamma$, then $I(\alpha)^{\princ_{I(\beta)}}=I(\alpha)=\gamma^{-1}I(\beta)$ and $I(\alpha)$ and $I(\beta)$ are in the same class;
\item if $I(\alpha)\nsubseteq(I(\beta):\gamma)$ for all units $\gamma$, then $X^{f-1}\in I(\alpha)^{\princ_{I(\beta)}}$.
\end{itemize}
Note that, if $f=1$, then $B$ is a field and thus in the second case $I(\alpha)\nsubseteq(I(\beta):\gamma)$ for all $\gamma\neq 0$, and thus $I(\alpha)^{\princ_{I(\beta)}}=B$.

Let $\mathcal{G}$ be a subset of $\mathcal{I}$ containing exactly one $I(\alpha)$ for each class with respect to $\sim$. If $I\in\mathcal{G}$, then $X^{f-1}\in I^{\princ_J}$ for every $J\in\mathcal{G}\setminus\{I\}$; thus, $X^{f-1}\in I^{\princ_{\mathcal{G}\setminus\{J\}}}$. Since $X^{f-1}\notin I$ by definition, by Lemma \ref{lemma:exponential} we have $|\insstar(k,B)|\geq 2^{|\mathcal{G}|}$.

We now need to estimate $|\mathcal{G}|$. By the reasoning above, the ideals $I(\alpha)$ and $I(\beta)$ can be in the same class only if $I(\beta)=\gamma I(\alpha)$ for some unit $\gamma$. Since $1\in I(\alpha)$, the ideal $I(\beta)$ can be in this form only if $\gamma^{-1}\in I(\beta)$, and thus there are at most $q^2-1$ possibilities; furthermore, $\gamma I(\alpha)=t\gamma I(\alpha)$ for all $t\in k\setminus\{0\}$, and thus the number of such ideals is at most $(q^2-1)/(q-1)=q+1$. Hence,
\begin{equation*}
|\mathcal{G}|\geq \inv{q+1}\left(\frac{q^{n-1}-1}{q-1}-1\right)=\frac{q^{n-1}-q}{q^2-1}\geq q^{n-3}.
\end{equation*}
The claim follows.
\end{proof}

The previous proof does not quite work in the non-local case, since in this case $B$ does not have an essential $B$-ideal analogous to $X^f$. A first attempt is the following.
\begin{prop}\label{prop:prodotto}
Let $B_1,B_2$ two $k$-algebras. There is an injective, order-preserving map 
\begin{equation*}
\insstar(k,B_1)\times\insstar(k,B_2)\hookrightarrow\insstar(k,B_1\times B_2).
\end{equation*}
In particular, $|\insstar(k,B_1)|\leq|\insstar(k,B_1\times B_2)|$.
\end{prop}
\begin{proof}
Let $\mathcal{G}_i:=\insfracid_0(k,B_i)$. For each $\star_1\in\insstar(k,B_1)$, let $\mathcal{H}_1:=\{I\times B_2\mid I\in\mathcal{G}_1^{\star_1}\}$ and define symmetrically $\mathcal{H}_2$. For every pair $(\star_1,\star_2)$ in the product $\insstar(k,B_1)\times\insstar(k,B_2)$, let $\iota(\star_1,\star_2)$ be the star operation generated by $\mathcal{H}_1\cup\mathcal{H}_2$: then, $\iota$ is a map from the product to $\insstar(k,B_1\times B_2)$. To show that it is injective, let $J:=I\times B_2$ for some $I\in\mathcal{G}_1$. Then, for every $L=L_1\times B_2\in\mathcal{H}_1$, we have
\begin{equation*}
(J:L)=(I\times B_2:L_1\times B_2)=(I:L_1)\times B_2
\end{equation*}
and thus $J^{\princ_L}=I^{\princ_{L_1}}\times B_2$; On the other hand, if $L=B_1\times L_2\in\mathcal{H}_2$, then $J^{\princ_L}=B_1\times B_2$. Therefore, $J^{\iota(\star_1,\star_2)}=I^{\star_1}\times B_2$; in the same way, if $J=B_1\times I$, then $J^{\iota(\star_1,\star_2)}=B_1\times I^{\star_2}$. Hence, $\iota$ is injective and the claim is proved.
\end{proof}

The previous proposition works best when $B$ has a local factor $B_i$ such that $n_i=e_if_i$ is large; in that case, using Proposition \ref{prop:down-numstar-loc} have $|\insstar(k,B)|\geq|\insstar(k,B_i)|\geq 2^{q^{n_i}-3}$, which gives $\theta(\mathbf{s},q)\geq n_i-3$. However, if every $n_i$ is small with respect to $n$ we need a different kind of estimate.

\begin{lemma}\label{lemma:numunits}
Let $B$ be a finite $k$-algebra, and let $n:=\ell_k(B)$. Then, for every $\epsilon>0$ there is a $q(\epsilon)$ such that if $q\geq q(\epsilon)$ the ring $B_\mathbf{s}$ has at least
\begin{equation*}
(q-1)^n\geq q^{(1-\epsilon)n}
\end{equation*}
units.
\end{lemma}
\begin{proof}
Let $B:=B_1\times\cdots\times B_t$, where $B_i:=\ins{F}_{q^{e_i}}[X]/(X^{f_i})$ for each $i$. Set $n_i:=e_if_i$. Each $B_i$ has $q^{e_if_i}-q^{e_i}\geq q^{n_i}-q^{n_i-1}$ units; thus the number of units of $B$ is at least
\begin{equation*}
\begin{aligned}
\prod_{i=1}^{t}(q^{n_i}-q^{n_i-1}) & \geq\prod_{i=1}^{t}q^{n_i-1}(q-1)=q^{n-(t-1)}(q-1)^{t}\geq\\
& \geq(q-1)^{n}=q^{n\log_q(q-1)}\geq q^{(1-\epsilon)n}
\end{aligned}
\end{equation*}
for large $q$. The claim is proved.
\end{proof}

\begin{prop}\label{prop:down-numstar-k}
Let $\mathbf{s}:=\{(e_1,f_1),\ldots,(e_t,f_t)\}$ be a sequence of pairs of positive integers, and let $n:=n(\mathbf{s})$ and $r:=e_tf_t$. If $n-r\geq 3$, then
\begin{equation*}
\liminf_{q\to\infty}\theta(\mathbf{s},q)\geq n-r-2.
\end{equation*}
\end{prop}
\begin{proof}
Since we are only interested in the limit for large $q$, we can suppose that $q>t$, so that we only need to consider subspaces in $\mathcal{F}_1(k,B)$.

Let $B_i:=\ins{F}_{q^{e_i}}[X]/(X^{f_i})$, write $B:=B_1\times\cdots\times B_t$ and $B':=B_1\times\cdots\times B_{t-1}$, so that $B=B'\times B_t$. By Lemma \ref{lemma:numunits}, for large $q$ the ring $B'$ has at least $q^{(1-\epsilon)(n-r)}$ units. For every unit $\alpha$ of $B'$ not belonging to $k$, let $I(\alpha):=\langle 1,(\alpha,0)\rangle$, and let $\mathcal{I}$ be the set of such $I(\alpha)$.

We note that the unique elements of $I(\alpha)$ in the form $(\gamma,0)$  are those with $\gamma=z\alpha$, for some $z\in k$: hence, the cardinality of $\mathcal{I}$ is equal to 
\begin{equation*}
\frac{|U(B')|}{q-1}\geq \frac{q^{(1-\epsilon)(n-r)}-1}{q-1}\geq q^{(1-\epsilon)(n-r)-1}
\end{equation*}
for large $q$. Furthermore, the second component of any element of $I(\alpha)$ belongs to $k$.

By Proposition \ref{prop:JstarI}, we have
\begin{equation*}
I(\alpha)^{\princ_{I(\beta)}}=\bigcap\{(I(\beta):\gamma)\mid \gamma I(\alpha)\subseteq I(\beta)\}.
\end{equation*}
We claim that if $\gamma\neq 0$ and $\gamma I(\alpha)\subseteq I(\beta)$ then $\gamma$ is a unit of $B$.

Indeed, let $\gamma:=(\gamma_1,\gamma_2)$. Since $\gamma I(\alpha)\subseteq I(\beta)$, we must have $\gamma\cdot 1\in I(\beta)$ and $\gamma(\alpha,0)\in I(\beta)$. The second component of $\gamma=\gamma\cdot 1$ is $\gamma_2$, and thus must belong to $k$. If $\gamma_2=0$, then $\gamma(\alpha,0)=(\gamma_1\alpha,0)$ and thus $\gamma_1\alpha=s\beta$ for some $s\in k$; in particular, $I(\alpha)=I(\beta)$.

If $\gamma_2\neq 0$ and $\gamma_1=0$, then $(0,\gamma_2)\in I(\beta)$; however, this would imply that $(1,0)=(1,1)-\gamma_2^{-1}\gamma\in I(\beta)$, against $\alpha\neq k$.

Suppose $\gamma_1,\gamma_2\neq 0$. Then, $\gamma_1\alpha=s\beta$ for some $s\in k$; since $\alpha$ is a unit, $s\beta\neq 0$ and thus $s\beta$ is a unit of $B'$, and thus the same must hold for $\gamma_1\alpha$; in particular, $\gamma_1$ must be a unit of $B'$, and thus $\gamma=(\gamma_1,\gamma_2)$ is a unit of $B$, as claimed.

Therefore, if $\gamma\neq 0$ and $\gamma I(\alpha)\subseteq I(\beta)$ then $(I(\beta):\gamma)=\gamma^{-1}I(\beta)$ has dimension $2$, and thus must be equal to $I(\alpha)$. Hence, $I(\alpha)^{\princ_{I(\beta)}}$ is either equal to $I(\alpha)$ or to $B$, and in the former case it must be in the form $\gamma^{-1}I(\beta)$ for some unit $\gamma$ of $B$ belonging to $I(\beta)$; in particular, $I(\alpha)$ and $I(\beta)$ must be in the same class. Let $\mathcal{G}$ be a set constructed by taking one representative for each class. There are at most $q^2-q$ possible units $\gamma$ (for each choice of $\alpha$), and they give $q$ different ideals $\gamma^{-1}I(\beta)$, since $\gamma^{-1}I(\beta)=(s\gamma)^{-1}I(\beta)$ for every $s\in k$. Hence,
\begin{equation*}
|\mathcal{G}|\geq\frac{q^{(1-\epsilon)(n-r)-1}}{q}=q^{(1-\epsilon)(n-r)-2}
\end{equation*}
for large $q$.

As in the proof of Proposition \ref{prop:down-numstar-loc}, $\mathcal{G}$ satisfies the hypothesis of Lemma \ref{lemma:exponential}, and thus $|\insstar(k,B)|\geq 2^{|\mathcal{G}|}$. Using the previous estimate and taking the limit as $q\to\infty$ we have our claim.
\end{proof}

\begin{prop}\label{prop:liminf}
Let $\mathbf{s}$ be a sequence of pair of positive integers, and let $n:=n(\mathbf{s})$. Then,
\begin{equation*}
\liminf_{q\to\infty}\theta(\mathbf{s},q)\geq\begin{cases}
\frac{n-4}{2} & \text{if~}n\text{~is even},\\
\frac{n-3}{2} & \text{if~}n\text{~is odd}.
\end{cases}
\end{equation*}
\end{prop}
\begin{proof}
Let $\mathbf{s}:=\{(e_1,f_1),\ldots,(e_t,f_t)\}$, and let $n_i:=e_if_i$ for each $i$. 

Suppose first that $n$ is even, and let $r$ be the largest of the $n_i$. If $r\leq n/2$, then by Proposition \ref{prop:down-numstar-k} we have, for large $q$,
\begin{equation*}
\liminf_{q\to\infty}\theta(\mathbf{s},q)\geq n-\frac{n}{2}-2=\frac{n}{2}-2=\frac{n-4}{2}.
\end{equation*}
On the other hand, if $r>n/2$ then $r\geq (n+2)/2$, and thus by Propositions \ref{prop:down-numstar-loc} and \ref{prop:prodotto} we have
\begin{equation*}
\theta(\mathbf{s},q)\geq r-3\geq\frac{n}{2}-2=\frac{n-4}{2},
\end{equation*}
and the claim follows.

Suppose now that $n$ is odd. If $t=1$ (i.e., if the associated Artinian ring $B$ is local) then by Proposition \ref{prop:down-numstar-loc} we have $\theta(\mathbf{s},q)\geq n-3\geq\frac{n-3}{2}$. If $t>1$, then there is at least one $n_i$ that is smaller than $n/2$; let it be $r$. Then, $r\leq(n-1)/2$ and, by Proposition \ref{prop:down-numstar-k}, we have
\begin{equation*}
\liminf_{q\to\infty}\theta(\mathbf{s},q)\geq n-\frac{n-1}{2}-2=\frac{n-3}{2}.
\end{equation*}
The claim is proved.
\end{proof}

\begin{teor}\label{teor:liminfsup}
Let $\mathbf{s}$ be a sequence of pair of positive integers, and let $n:=n(\mathbf{s})$. Then,
\begin{equation*}
\begin{dcases}
\frac{n-4}{2}\leq \liminf_{q\to\infty}\theta(\mathbf{s},q)\leq\limsup_{q\to\infty}\theta(\mathbf{s},q)\leq\frac{n(n-2)}{4} & \text{if~}n\text{~is even},\\
\frac{n-3}{2}\leq \liminf_{q\to\infty}\theta(\mathbf{s},q)\leq\limsup_{q\to\infty}\theta(\mathbf{s},q)\leq\frac{(n-1)^2}{4} & \text{if~}n\text{~is odd}.
\end{dcases}
\end{equation*}
\end{teor}
\begin{proof}
The first inequality (of both cases) follows from Proposition \ref{prop:liminf}, and last from Theorem \ref{teor:up-numstar}.
\end{proof}

Note that the limit inferior of the previous theorem is only significant for $n\geq 5$. For $n\leq 3$ we can apply Proposition \ref{prop:ns3}, while we will analyze the case $n=4$ in Section \ref{sect:n4}.

\section{Codimension 1}\label{sect:codim1}
In this section, we consider the subspaces of $B$ having codimension $1$; in particular, we want to show that the number of classes of these ideals (with respect to generating the same star operation) can be bounded above  by a function that does \emph{not} depend on the size of $k$.

The main tool is the use of canonical ideals. Given an integral domain $D$ with quotient field $K$, an ideal $I$ of $D$ is a \emph{canonical ideal} of $D$ if, for every ideal $J$, we have $J=(I:_K(I:_KJ))$; note that there are many equivalent (and more general) definitions (see e.g. \cite[Chapter 15]{matlis-1dimCM}), but we use this one since it is closer to our subject. In particular, we need the following characterization.
\begin{prop}\label{prop:canonico}
Let $(D,\mathfrak{m})$ be a one-dimensional Noetherian local domain, let $T$ be its integral closure, and suppose that $(D:T)=\mathfrak{m}$. Let $I\in\insfracid_0(D,T)$. Then, $I$ is a canonical ideal if and only if $\ell_D(T/I)=1$ and $\mathfrak{m}$ is the largest ideal of $T$ contained in $I$.
\end{prop}
\begin{proof}
By \cite[Theorem 15.5]{matlis-1dimCM}, $I$ is a canonical ideal if and only if $\ell_D((I:\mathfrak{m})/I)=1$.

Since $\mathfrak{m}$ is a $T$-ideal contained in $I$, we have $I\subsetneq T\subseteq(I:\mathfrak{m})$; hence, $\ell_D((I:\mathfrak{m})/I)=1$ if and only if $(I:\mathfrak{m})=T$ and $\ell_D(T/I)=1$. Furthermore, if $L$ is a $T$-ideal contained in $I$, then $(I:\mathfrak{m})\supseteq(L:\mathfrak{m})$, which is greater than $T$ as soon as $\mathfrak{m}\subsetneq L$. Hence, $\mathfrak{m}$ must be the largest ideal of $T$ contained in $I$.

Conversely, if the two conditions hold, we claim that $(I:\mathfrak{m})=T$: otherwise, there would be $t\in(I:\mathfrak{m})\setminus T$, but then $t\mathfrak{m}$ would be a $T$-ideal contained in $I$ but not in $\mathfrak{m}$, against the hypothesis. Therefore $\ell_D((I:\mathfrak{m})/I)=\ell_D(T/I)=1$ is a canonical ideal.
\end{proof}

Suppose now we are in our setting: $k$ is a field and $B$ is a finite $k$-algebra. Following \cite{multiplicative}, we say that $I\in\inssubmod_k(B)$ is a \emph{canonical ideal} of $(k,B)$ if $\princ_I$ is the identity; this definition mirrors the definition of \emph{$m$-canonical ideals} given in \cite{hhp_m-canonical}. By Proposition \ref{prop:JstarI} and the explicit description of the closure generated by $I$ (see \cite[Lemma 4.4]{multiplicative}), for $I\in\mathcal{F}_1(k,B)$ this is equivalent to the condition that $(I:_B(I:_BJ))=J$ for all $J\in\mathcal{F}_1(k,B)\setminus\{k\}$. 
\begin{prop}\label{prop:can-dim0}
Let $(D,\mathfrak{m})$ be a one-dimensional Noetherian local domain, let $T$ be its integral closure and  suppose $(D:T)=\mathfrak{m}$. Suppose $I$ is a fractional ideal of $D$ satisfying $(D:T)\subseteq I\subseteq T$. Then, $I$ is a canonical ideal of $D$ if and only if its image $J$ is a canonical ideal of $(D/\mathfrak{m},T/\mathfrak{m})$.
\end{prop}
\begin{proof}
By \cite[Theorem 6.4 and Corollary 6.5]{multiplicative} (see also the discussion after Definition \ref{defin:star}), there is a bijection between the star operations on $D$ and the star operations on $(D/\mathfrak{m},T/\mathfrak{m})$; in particular, the star operation generated by $I$ correspond to the star operation generated by its image $J$. The claim follows from the definitions.
\end{proof}

\begin{cor}\label{cor:canonico}
Let $B$ be a finite $k$-algebra that is a principal ideal ring, and let $I\in\insfracid_1(k,B)$. Then, $I$ is a canonical ideal of $(k,B,\insfracid_1(A,B))$ if and only if $\ell_k(B/I)=1$ and $I$ does not contain any nonzero $B$-ideal.
\end{cor}
\begin{proof}
By Proposition \ref{prop:portasu}, we can find a principal ideal domain $T$ such that $B$ is a quotient of $T$. Setting $D$ to be the counterimage of $k$, we are in the setting of Proposition \ref{prop:can-dim0}, and then we can apply Proposition \ref{prop:canonico}. 
\end{proof}

Given a $k$-algebra $B$, let now $\mathcal{H}(B)$ be the set of all $k$-hyperplanes of $B$ containing $1$, i.e., the set of $I\in\insfracid_1(k,B)$ such that $\ell_k(B/I)=1$. We want to estimate in how many different star operations the elements of $\mathcal{H}(B)$ generate, i.e., in how many classes $\mathcal{H}(B)$ is partitioned.

If $I\in\insfracid(A,B)$, we define $Z(I)$ to be the largest $B$-ideal contained in $I$: this is well-defined since if $L_1,L_2\subseteq I$ are $B$-ideals then also $L_1+L_2$ is a $B$-ideal contained in $I$.
\begin{lemma}\label{lemma:JZI}
Let $I,J\in\insfracid(A,B)$, and let $\princ_I$ be the multiplicative operation generated by $I$. Then, $J+Z(I)\subseteq J^{\princ_I}$.
\end{lemma}
\begin{proof}
Since $J\subseteq J^{\princ_I}$, we only need to show that $Z(I)\subseteq J^{\princ_I}$. If $t\in Z(I)$, then $tT\subseteq Z(I)\subseteq I$; in particular, $t(I:_BJ)\subseteq I$ and thus $t\in(I:_B(I:_BJ))=J^{\princ_I}$ (the last equality coming from \cite[Lemma 4.4]{multiplicative} and Proposition \ref{prop:JstarI}.
\end{proof}

\begin{prop}\label{prop:classi-iperpiani}
Let $B$ be a finite $k$-algebra that is a principal ideal ring, and let $n:=\ell_k(B)$. Then, the following hold.
\begin{enumerate}[(a)]
\item\label{prop:classi-iperpiani:IJ} If $I,J\in\mathcal{H}(B)$, then $\princ_I=\princ_J$ if and only if $Z(I)=Z(J)$. 
\item\label{prop:classi-iperpiani:num} If $\ell_k(B)=n$, then $\mathcal{H}(B)$ is divided into at most $2^n$ classes.
\end{enumerate}
\end{prop}
\begin{proof}
\ref{prop:classi-iperpiani:IJ} If $Z(I)\neq Z(J)$, then without loss of generality $Z(I)\nsubseteq Z(J)$ and thus $Z(I)\nsubseteq J$. By Lemma \ref{lemma:JZI}, $B=J+Z(I)\subseteq J^{\princ_I}$, so that $J\neq J^{\princ_I}$ and $I$ and $J$ are not in the same class.

Conversely, suppose $Z(I)=Z(J)$, and consider $B':=B/Z(I)$. Then, $I/Z(I)$ and $J/Z(I)$ are both hyperplanes of $B'$ not containing any $B'$-ideal; by Corollary \ref{prop:can-dim0}, they are canonical ideals of $B'$. Hence, they generate the identity star operation on $B'$, and so in $B$ they generate the map $\princ_I:L\mapsto L+Z(I)$. In particular, $I$ and $J$ are in the same class.

\ref{prop:classi-iperpiani:num} Let $B=B_1\times\cdots\times B_t$. The ideals of $B$ are in the form $I_1\times\cdots\times I_t$, where each $I_i$ is either $B_i$ or an ideal of $B_i$. Since $B_i\simeq L_i[X]/(X^{f_i})$ for some extension $L_i$ of $k$, there are $f_i+1$ possible $I_i$, and thus the number of ideals of $B$ is $(f_1+1)\cdots(f_t+1)$. However, 
\begin{equation*}
\prod_{i=1}^t(f_i+1)\leq\left(\frac{(f_1+1)+\cdots+(f_t+1)}{t}\right)^t\leq\left(\frac{n+t}{t}\right)^t=\left(1+\frac{n}{t}\right)^t
\end{equation*}
using the inequality between geometric and arithmetic mean and the fact that $f_1+\cdots+f_t\leq n$. Furthermore, the function $t\mapsto\left(1+\frac{n}{t}\right)^t$ is increasing in $t$, and thus, since $t\leq n$, the number of classes is at most $(1+1)^n=2^n$, as claimed.
\end{proof}

\begin{oss}
~\begin{enumerate}
\item The bound of Proposition \ref{prop:classi-iperpiani} does not depend on the size of $k$.
\item Not all ideals of $B$ are equal to $Z(I)$ for some ideal $I$ of codimension $1$: for example, $B$ itself doesn't. Another case are the subspaces that are themselves hyperplanes: in that case, we should have $I=Z(I)$, which is impossible if $1\in I$.
\item If $n=3$, then the bound of Proposition \ref{prop:classi-iperpiani} gives $|\insstar(k,B)|\leq 65$ for all choices of $B$. By considering in more details the different cases, and excluding the impossible $B$-ideals, it is possible to obtain much tighter bounds; for example, if $B=k^3$ then we have only three possible $Z(I)\neq(0)$ (the ones in the form $I_1\times I_2\times I_3$ with one $I_i$ equal to $k$ and two $I_j$ equal to $0$), and thus $|\insmult_0(k,B)|\leq 1+2^3=9$, which is the exact value of $|\insstar(k,B)|$ by \cite[Theorem 3.1(2)]{houston_noeth-starfinite} and Proposition \ref{prop:ns3}.
\end{enumerate}
\end{oss}

\section{The case $n=4$}\label{sect:n4}
In this section, we are going to prove Conjecture \ref{cong:ns} for $n(\mathbf{s})=4$. In this case, Theorem \ref{teor:liminfsup} gives
\begin{equation*}
0\leq\liminf_{q\to\infty}\theta(\mathbf{s},q)\leq\limsup_{q\to\infty}\theta(\mathbf{s},q)\leq 2,
\end{equation*}
while Proposition \ref{prop:gamman-PVD} says that the limit of $\theta((n,1),q)$ is $1$. Hence, we need to prove that $\theta(\mathbf{s},q)\to 1$ for every $\mathbf{s}$. 

We start from the upper bound.
\begin{lemma}\label{lemma:num-sol}
Let $k$ be a finite field of cardinality $q$, let $\mathbf{s}:=\{(e_1,f_1),\ldots,(e_t,f_t)\}$ be a sequence of pairs of positive integers and let $B$ be the Artinian ring associated to $\mathbf{s}$. Let $e:=\sum_ie_i$ and $n:=\ell_k(B)=\sum_ie_if_i$. Let $\Sigma:=\{\alpha\in B\mid \alpha^2+r\alpha+s=0$ for some $r,s\in k\}$. Then,
\begin{equation*}
|\Sigma|\leq q^{n-e+1}+2^t(q^2-q)\leq q^{n-t+1}+2^t(q^2-q)
\end{equation*}
\end{lemma}
\begin{proof}
Let $B=B_1\times\cdots\times B_t$, where each $B_i$ is a local $k$-algebra, and let $n_i:=\ell_k(B_i)$. Let $\pi_i:B_i\longrightarrow L_i$ be the quotient of $B_i$ on its residue field $L_i$.

Let $r,s\in k$, and consider the equation $f(X)=X^2+rX+s$. Let $S_i(f)$ be the set of solutions of $f(X)=0$ in $B_i$, and suppose that $S_i\neq\emptyset$. If $\alpha\in S_i$, then we can factorize $f(X)$ as $(X-\alpha)(X-\beta)$ for some $\beta\in B_i$. We distinguish two cases.

If $f(X)$ has simple roots in the algebraic closure $\overline{k}$ of $k$, then the images of $\alpha$ and $\beta$ in the residue field of $B_i$ are distinct. Hence, for every $\gamma\in S_i$ one of $\pi_i(\gamma-\alpha)$ and $\pi_i(\gamma-\beta)$ is a unit, and thus $(\gamma-\alpha)(\gamma-\beta)=0$ implies that $\gamma=\alpha$ or $\gamma=\beta$. Thus, $|S_i(f)|\leq 2$.

If $f(X)$ has a double root over $\overline{k}$, then all roots of $f(X)$ in $B_i$ must belong to $\pi_i^{-1}(\overline{\alpha})$; since the latter is a coset of an ideal of $B_i$, we have $|S_i(f)|\leq q^{e_i(f_i-1)}=q^{n_i-e_i}$.

Let now $S(f)$ be the set of solutions of $f(X)=0$ in $B$. Then, $S(f)=S_1(f)\times\cdots\times S_t(f)$; thus, if $f(X)$ has simple roots in $\overline{k}$ then $|S(f)|\leq 2^t$, while if $f(X)$ has a double root then $|S(f)|\leq \prod_iq^{n_i-e_i}=q^{n-e}$.

There are exactly $q$ polynomials of degree $2$ with a double root (since $k$ is finite, hence perfect), and thus $q^2-q$ polynomials with single roots. Hence, 
\begin{equation*}
|\Sigma|\leq(q^2-q)\cdot 2^t+q\cdot q^{n-e}=q^{n-e+1}+2^t(q^2-q)\leq q^{n-t+1}+2^t(q^2-q),
\end{equation*}
as claimed.
\end{proof}

\begin{prop}\label{prop:up-numstar-4}
Suppose $n(\mathbf{s})=4$. Then,
\begin{equation*}
\limsup_{q\to\infty}\theta(\mathbf{s},q)\leq 1.
\end{equation*}
\end{prop}
\begin{proof}
Let $B$ be the ring associated to $\mathbf{s}$, and let $\mathcal{G}:=\mathcal{F}_1(k,B)$. Let $\mathcal{G}_i:=\{I\in\mathcal{G}\mid \dim_kI=i\}$. Then, $\mathcal{G}_0$ is empty, while $\mathcal{G}_1$ and $\mathcal{G}_4$ are both singletons: the former is $\{k\}$, the latter is $\{B\}$. Both are contained in $\mathcal{G}^\star$ for every star operation $\star$ on $(k,B)$, and thus $|\insstar(k,B)|\leq 2^{|[\mathcal{G}_2]|+|[\mathcal{G}_3]|}$.

By Proposition \ref{prop:classi-iperpiani}, $|[\mathcal{G}_3]|\leq 4$.

Consider now $\mathcal{G}_2$. Then, 
\begin{equation*}
|\mathcal{G}_2|=Z_1(q,3)=\qbin{3}{1}=\frac{q^3-1}{q-1}=q^2+q+1.
\end{equation*}
Let $\mathcal{I}$ be the set of the subspaces in the form $\langle 1,\alpha\rangle$, where $\alpha$ is a unit of $B$ not belonging to $k$; clearly, $\mathcal{I}\subseteq\mathcal{G}_2$. By Lemma \ref{lemma:numunits}, $B$ has at least $(q-1)^4$ units; furthermore, every $I(\alpha)$ contains at most $q^2-q$ of them (not counting those in $k$). Thus,
\begin{equation*}
|\mathcal{I}|\geq\frac{(q-1)^4-q}{q^2-q}=\frac{(q-1)^3-1}{q}=q^2-3q+3-\frac{2}{q}\geq q^2-3q+3.
\end{equation*}
In particular, $|\mathcal{G}_2\setminus\mathcal{I}|\leq q^2+q+1-(q^2-3q+3)=4q-2$.

Let $\alpha$ be a unit. Then, $\alpha^{-1}I(\alpha)=\alpha^{-1}\langle 1,\alpha\rangle=\langle 1,\alpha^{-1}\rangle=I(\alpha^{-1})$. We have that $I(\alpha)=I(\alpha^{-1})$ if and only if $\alpha^{-1}\in\langle 1,\alpha\rangle$, that is, if and only if $\alpha$ satisfies an equation of degree $2$ with coefficients in $k$. Since every element of $I(\alpha)$ is in the form $s+t\alpha$ for some $s,t\in k$, we have that $\alpha^{-1}I(\alpha)=I(\alpha)$ if and only if $\beta^{-1}I(\alpha)=I(\alpha)$ for some unit $\beta\in I(\alpha)$. Therefore, $\mathcal{I}$ can be partitioned into two classes:
\begin{itemize}
\item $\mathcal{I}_0:=\{I(\alpha)\mid \alpha^{-1}I(\alpha)=I(\alpha)\}$;
\item $\mathcal{I}_1:=\{I(\alpha)\mid \alpha^{-1}I(\alpha)\neq I(\alpha)\}$;
\end{itemize}

Each element of $\mathcal{I}_1$ contains at least $q^2-2q$ units not in $k$ ($q^2$ elements, minus $q$ in $k$, minus at most $q$ that are not units). Thus, for each subspace there are at least $(q^2-2q)/(q-1)=q-1+\frac{1}{q-1}\geq q-2$ equivalent subspaces, and thus (applying again Lemma \ref{lemma:num-sol}) the number of classes is at most
\begin{equation*}
|[\mathcal{I}_1]|\leq\frac{q^2+q+1-(q+16)}{q-2}\leq q+4.
\end{equation*}

Each subspace in $\mathcal{I}_0$ is equivalent only to himself; to estimate its cardinality, we distinguish two cases.

If $t\geq 2$, by Lemma \ref{lemma:num-sol} there are at most $q^{n-t+1}+2^t(q^2-q)$ elements that are solutions of an equation of degree $2$; hence, they can fill at most
\begin{equation*}
\frac{q^{n-t+1}+2^t(q^2-q)}{q^2-2q}=\frac{q^{n-t}+2^t(q-1)}{q-1}\geq q^{n-t-1}+2^t\geq q+16
\end{equation*}
ideals, i.e., $|\mathcal{I}_0|\leq q+16$. Therefore,
\begin{equation*}
|\insstar(k,B)|\leq 2^{[\mathcal{G}_3]+[\mathcal{G}_2\setminus\mathcal{I}]+[\mathcal{I}_0]+[\mathcal{I}_1]}\leq 2^{4+4q+2+q+16+q+4}=2^{6q+26}
\end{equation*}
and
\begin{equation*}
\limsup_{q\to\infty}\theta(\mathbf{s},q)\leq\log_q\log_2(2^{6q+26})=\log_q(6q+26)=1,
\end{equation*}
as claimed.

Suppose $t=1$. Then, $\mathbf{s}=(e,f)$ with $ef=4$, and thus we have four possibilities: $(4,1)$, $(2,2)$, $(1,4)$. The first one is exactly the one considered in Proposition \ref{prop:gamman-PVD}, and thus the claim holds. The last one corresponds to the star operations on the ring $A:=k[[X^4,X^5,X^6,X^7]]$, for which $|\insstar(A)|=2^{2q+1}+2^{q+1}+2$ (see \cite{white-tesi-sgr}), and thus the limit of $\theta(\mathbf{s},q)$ is again $1$.

The case $\mathbf{s}=\{(2,2)\}$ correspond to $B=L[X]/(X^2)$, where $L$ is the field of cardinality $q^2$. By Lemma \ref{lemma:num-sol}, we have that the number of solutions to equations of degree $2$ is at most $q^{4-2+1}+2(q^2-q)=q^3+2(q^2-q)$. Reasoning as above, we get $|\mathcal{I}_0|\leq q+5$, and so
\begin{equation*}
|\insstar(k,B)|\leq 2^{q^{4+4q+2+q+5+q+4}}=2^{6q+15}.
\end{equation*}
Hence, $\limsup_q\theta(\mathbf{s},q)\leq 1$, as claimed.
\end{proof}

To study the limit inferior, we need to reason by cases. Let $\mathbf{s}=\{(e_1,f_1),\ldots,(e_t,f_t)\}$: if $n_i:=e_if_i$, then the set $\{n_1,\ldots,n_t\}$ is a partition of $n:=n(\mathbf{s})$. In particular, if $n=4$ then we have five possible partitions:
\begin{itemize}
\item $4$;
\item $3+1$;
\item $2+1+1$;
\item $1+1+1+1$;
\item $2+2$.
\end{itemize}
Note that each of the previous partition may correspond to more than one ring $B$. For example, the case $3+1$ is further subdivided into the cases $\mathbf{s}=\{(3,1),(1,1)\}$ (corresponding to $B=L\times k$, where $L$ is an extension of $k$ of degree $3$) and $\mathbf{s}=\{(1,3),(1,1)\}$ (corresponding to $B=k[X]/(X^3)\times k$).

\begin{prop}\label{prop:n4-4}
Preserve the notation above. If $n_1=4$, then $\liminf_q\theta(\mathbf{s},q)\geq 1$.
\end{prop}
\begin{proof}
Apply Proposition \ref{prop:down-numstar-loc}: $\mathbf{s}=\{(e,f)\}$ and thus $\theta((e,f),q)\geq 4-3=1$.
\end{proof}

\begin{prop}\label{prop:n4-r1}
Preserve the notation above. If $e_i=f_i=1$ for some $i$, then $\liminf_q\theta(\mathbf{s},q)\geq 1$.
\end{prop}
\begin{proof}
We can apply Proposition \ref{prop:down-numstar-k} with $r=1$, obtaining $\liminf_q\theta(\mathbf{s},q)\geq 4-1-2=1$.
\end{proof}

The only case remaining is the partition $2+2$. To examine it, we apply a variant of the proof of Proposition \ref{prop:down-numstar-loc}.
\begin{prop}\label{prop:n4-22}
Suppose $\mathbf{s}=\{(e_1,f_1),(e_2,f_2)\}$ with $e_1f_1=e_2f_2=2$. Then,
\begin{equation*}
\liminf_{q\to\infty}\theta(\mathbf{s},q)\geq 1.
\end{equation*}
\end{prop}
\begin{proof}
Let $B=B_1\times B_2$ be the ring associated to $\mathbf{s}$. 

Suppose first that $B_1=L$ is the extension of degree $2$ of $k$. Then, $B_1$ has $q^2-q$ units not belonging to $k$, while $B_2$ has at least $q^2-2q$ of them: hence, $B$ has at least $q^2(q-1)(q-2)$ of these unit. For every such unit $\alpha$ of $B$, let $I(\alpha):=\langle 1,\alpha\rangle$, and let $\mathcal{I}$ be the set of all $I(\alpha)$. No element of $I(\alpha)$ (different from $0$) has a component equal to $0$: otherwise, setting $\alpha=(\alpha_1,\alpha_2)$, we would have $s+r\alpha_i=0$ for some $s,r\in k$ not both $0$, against the fact that $\alpha_i$ is not in $k$. Therefore, each there are $q^2-q$ elements $\beta$ such that $I(\alpha)=I(\beta)$, and so $\mathcal{I}$ has at least $\frac{q^2(q-1)(q-2)}{q^2-q}=q(q-2)$ elements.

Let $\gamma=(\gamma_1,\gamma_2)\in I(\alpha)$: we claim that $J:=(I(\alpha):\gamma)$ cannot have dimension $3$ over $k$. 

If $\gamma_1=0$, then $(0,\gamma_2\alpha_2)\in I(\alpha)$: then we must have $\gamma_2\alpha_2=0$, and since $\alpha_2$ is a unit, this implies that $\gamma_2=0$, i.e., that $\gamma=0$ and thus $J=B$. If $\gamma_1\neq 0$, then $\gamma_1$ is a unit of $L$. If $\dim_kJ>2$, then, $J\cap(B_1\times\{0\})$ must be nontrivial, i.e., there would be some $\beta:=(\beta_1,0)\neq 0$ such that $\beta\in J$. However, $\beta\gamma=(\beta_1,0)(\gamma_1,\gamma_2)=(\beta_1\gamma_1,0)\neq(0,0)$ since $\gamma_1$ is a unit and $\beta_1\neq 0$. Hence, if $\gamma\neq 0$ then $\dim_kJ=2$, while if $\gamma=0$ then $J=B$ and $\dim_kJ=4$.

Hence, $I(\alpha)^{\princ_{I(\beta)}}$ is equal either to $I(\alpha)$ (if $\gamma^{-1}I(\alpha)=I(\beta)$ for some unit $\gamma\in I(\alpha)$) or to $B$; in the former case $I(\alpha)$ and $I(\beta)$ are in the same class, in the latter they belong to different classes. As in the proof of Proposition \ref{prop:down-numstar-loc}, we have $q^2-1$ units $\gamma$ in $I(\alpha)$, and if $t\in k\setminus\{0\}$ then $(t\gamma)^{-1}I(\alpha)=\gamma^{-1}I(\alpha)$, and thus each class contains at most $(q^2-1)/(q-1)=q+1$ ideals of this kind.

Let $\mathcal{G}$ be a set of representatives of the classes of $\mathcal{I}$: then,
\begin{equation*}
[\mathcal{G}]\geq \frac{|\mathcal{I}|}{q+1}\geq\frac{q(q-2)}{q+1}\geq q-3.
\end{equation*}
By Lemma \ref{lemma:exponential}, it follows that $|\insstar(k,B)|\geq 2^{q-3}$, and thus $\liminf_q\theta(\mathbf{s},q)\geq\liminf_q\log_q(q-3)=1$, as claimed.

\medskip

Clearly, if $B_2=L$ then the argument is identical. Suppose thus that $B_1\neq L\neq B_2$: then, $B_1$ and $B_2$ must be isomorphic to $k[X]/(X^2)$. Each $B_i$ has $q^2-2q$ units not belonging to $k$; we choose those units $\alpha=(\alpha_1,\alpha_2)$ such that the image of $\alpha_1$ and $\alpha_2$ in $k$ are distinct. Thus, we obtain $(q^2-2q)(q^2-3q)=q^2(q-2)(q-3)$ possible $\alpha$. The rest of the reasoning proceeds in the same way, obtaining a bound $|\mathcal{G}|\geq q-c$ for some constant $c$, from which the claim follows.
\end{proof}

\begin{teor}\label{teor:n4}
Let $\mathbf{s}$ be a sequence of pairs of positive integers with $n:=n(\mathbf{s})=4$. Then,
\begin{equation*}
\lim_{q\to\infty}\theta(\mathbf{s},q)=1.
\end{equation*}
\end{teor}
\begin{proof}
By Proposition \ref{prop:up-numstar-4}, $\limsup_q\theta(\mathbf{s},q)\leq 1$. By Propositions \ref{prop:n4-4}, \ref{prop:n4-r1} and \ref{prop:n4-22} (and the discussion before Proposition \ref{prop:n4-4}), on the other hand, $\limsup_q\theta(\mathbf{s},q)\geq 1$. Hence, the limit exists and is equal to $1$.
\end{proof}

The previous theorem does not allow by itself an explicit determination of the number of star operations on $(k,B)$, mainly because for every $\mathbf{s}$ it is necessary to consider a different number of special cases, which give a different counting. (However, if one follows the proofs in a special case it is possible to obtain some more explicit estimates.) In two cases, the cardinality of $\insstar(k,B)$ has been determined explicitly:  
\begin{itemize}
\item if $\mathbf{s}=\{(4,1)\}$, then $B=F$ (where $F$ is the extension of $k$ of degree $4$) and the star operations on $(k,B)$ correspond to the star operations on the pseudo-valuation domain $A:=k+XF[[X]]$; in this case,
\begin{equation*}
|\insstar(A)|=2^{q+1}+1
\end{equation*}
by the results in \cite{cardinality_pvd} and \cite{pvd-star};
\item if $\mathbf{s}=\{(1,4)\}$ then $B=k[X]/(X^4)$, or in domain terms we are in the residually rational case $A:=k[[X^4,X^5,X^6,X^7]]$; in this case,
\begin{equation*}
|\insstar(A)|=2^{2q+1}+2^{q+1}+2
\end{equation*}
by \cite{white-tesi-sgr}.
\end{itemize}

We now show how to obtain a precise counting in a further special case, namely when $B=L\times k$ and $L$ is the extension of $k$ of degree $3$.
\begin{ex}
Let $B:=L\times k$, and let $\mathcal{G}_2$ and $\mathcal{G}_3$ be, respectively, the set of $k$-subspaces of $B$ having dimension $2$ and $3$ over $k$ that contain $1$.

Let $\mathcal{G}:=\mathcal{F}_0(k,B)=\mathcal{F}_1(k,B)$ (in this case $t=2$ and thus the two sets are equal). On the set $[\mathcal{G}]$ of classes, define $[I]\preceq[J]$ if and only if $\princ_I\geq\princ_J$, that is, if and only if $I=I^{\princ_J}$. Then, $\preceq$ is a partial order on $\mathcal{G}$, and the set $[\mathcal{G}]^\star:=\{[I]\mid I^\star=I\}$ is a downset for every $\star\in\insstar(k,B)$. See \cite[Definition 4.7 onwards]{multiplicative} for more information about this order.

As in the proof of Proposition \ref{prop:up-numstar-4}, let $\mathcal{G}_i$ be the set of $I\in\insfracid_1(k,B)$ having dimension $i$ over $k$.

By Proposition \ref{prop:classi-iperpiani}\ref{prop:classi-iperpiani:IJ}, the class of $I\in\mathcal{G}_3$ depends only on the largest $B$-ideal $Z(I)$ contained into $I$. There are four ideals of $B$: the zero ideal, $\{0\}\times B_2$, $B_1\times\{0\}$ and $B_1\times B_2$. In particular, the latter two cannot be in the form $Z(I)$, since $B=B_1\times B_2$ has dimension $4$, while $B_1\times\{0\}$ has dimension $3$ and does not contain $1$. Hence, we have two classes:
\begin{itemize}
\item $\mathcal{G}_{3,\mathrm{can}}:=\{I\mid Z(I)=\{(0,0)\}\}$: these are the canonical ideals;
\item $\mathcal{G}_{3,1}:=\{I\mid Z(I)=(0)\times B_2\}$.
\end{itemize}
In particular, the elements of $\mathcal{G}_{3,\mathrm{can}}$ close every $k$-subspace, and thus $\mathcal{G}_{3,\mathrm{can}}$ is the maximum of $([\mathcal{G}],\preceq)$. On the other hand, the elements of $\mathcal{G}_{3,1}$ close exactly the subspaces $J$ such that $((0)\times B_2)J\subseteq J$, and the only subspace satisfying this condition is $R(1):=\langle (1,1),(1,0)\rangle=k\times k$. Note that $R(1)$ does not close any other subspace different from $k$, because there is no other subspace containing $(0)\times B_2$. Thus, we have the following partial graph of $[\mathcal{G}]$:
\begin{equation*}
\mathcal{G}_{3,\mathrm{can}}\to \mathcal{G}_{3,1}\to [R(1)]\to [k].
\end{equation*}

Consider now $I\in\mathcal{G}_2$. Then, we can write $I=\langle(1,1),(\alpha,\beta)\rangle$ for some $(\alpha,\beta)\in B$. Since $\beta\in k$, we can suppose without loss of generality that $\beta=0$, i.e., that $I=R(\alpha):=\langle(1,1),(\alpha,0)\rangle$. If $\alpha\in k$ then we have $R(\alpha)=R(1)$ (see above) and thus it is enough to consider the case $\alpha\notin k$.

There are $q^3-q$ of such $\alpha$, and each subspace can be generated by $q-1$ of them; hence, we obtain $q^2+q$ ideals in this form. By the proof of Proposition \ref{prop:down-numstar-k}, every equivalence class of them contains $q$ subspaces, and they are not comparable under $\preceq$. Furthermore, by direct calculation $(I(\alpha):(\alpha,0))=R(1)$ (so $[R(1)]\preceq[I(\alpha)]$) for every $\alpha$, while no element of $\mathcal{G}_{3,1}$ is closed by $I(\alpha)$. Thus, we obtain the following complete picture of $[\mathcal{G}]$:
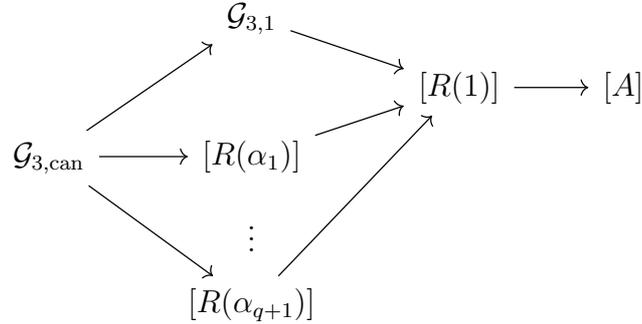
\begin{equation*}
\begin{tikzcd}[row sep=tiny]
 & \mathcal{G}_{3,1}\arrow{dr}\\
& & {[R(1)]}\arrow[r] & {[A]}\\
\mathcal{G}_{3,\mathrm{can}}\arrow[uur]\arrow{r}\arrow{rdd} & {[R(\alpha_1)]}\arrow{ur}\\
& \vdots\\
& {[R(\alpha_{q+1})]}\arrow[uuur]
\end{tikzcd}
\end{equation*}
Hence, there are $2^{q+2}+2$ nonempty downsets:
\begin{itemize}
\item the whole $[\mathcal{G}]$;
\item $2^{q+2}-1$ downsets in the form $X\cup\{[A],R[1]\}$, for nonempty $X\subseteq\{\mathcal{G}_{3,1},R(\alpha_1),\ldots,R(\alpha_{q+1})\}$;
\item $[A]$ and $[A]\cup[R(1)]$.
\end{itemize}
By direct inspection and by \cite[Proposition 4.10]{multiplicative}, all these downset give rise to a star operation, and thus $|\insstar(k,B)|=2^{q+2}+2$ star operations.
\end{ex}

We note that, in the three cases considered explicitly for $n=4$, the number of star operations is a polynomial $f(X)$ evaluated in $X=2^q$: indeed, for $\mathbf{s}=\{(4,1)\}$ the polynomial is $f(X)=2X+1$, for $\mathbf{s}=\{(1,4)\}$ we have $f(X)=2X^2+2X+2$, and for $\mathbf{s}=\{(3,1),(1,1)\}$ we have $f(X)=4X+2$. This lead to the following, final
\begin{congett}\label{cong:l4}
For every $\mathbf{s}$ such that $n(\mathbf{s})=4$ there is a polynomial $f(X)\inZ[X]$ such that $\Lambda(\mathbf{s},q)=f(2^q)$ for every $q$.
\end{congett}

This conjecture cannot be generalized to higher lengths: for example, in the case $\mathbf{s}=\{(5,1)\}$, the number of star operations is equal to $(q^2+5)2^{q^2}-(q^2-1)$ \cite[Theorem 4.3]{pvd-star}.

\bibliographystyle{plain}
\bibliography{/bib/articoli,/bib/libri,/bib/miei}
\end{document}